\documentclass{amsart}
\usepackage{amsthm}
\usepackage{amsmath}
\usepackage{amssymb}
\usepackage{amsfonts}
\usepackage{graphicx}
\usepackage{hyperref}
\usepackage{cite}
\usepackage{tikz}
\usepackage{comment}
\usepackage{booktabs}   
\usepackage{caption}    
\usepackage{array}      
\usepackage{tabularx}   
\newtheorem{theorem}{Theorem}[section] 
\newtheorem{lemma}[theorem]{Lemma}    
\newtheorem{corollary}[theorem]{Corollary} 
\newtheorem{proposition}[theorem]{Proposition} 
\newtheorem{definition}[theorem]{Definition} 
\newtheorem{remark}[theorem]{Remark}         
\newtheorem{example}[theorem]{Example}
\newcommand{\dz}{d^+_i}
\newcommand{\df}{d^-_i}
\newcommand{\Le}{L_{+}-\epsilon L_{-}}
\newcommand{\LE}{L_{\epsilon}}
\newcommand{\tw}{\tilde{w}^{\epsilon}}
\newcommand{\Om}{\Omega_{\epsilon}^{\sigma}}

\title{repelling curvature via $\epsilon-$repelling Laplacian on positive connected signed graphs}
\author{Yong Lin}
\address{Tsinghua University}
\email{yonglin@tsinghua.edu.cn}
\author{Shi Wan}
\address{Tsinghua University}
\email{wans21@mails.tsinghua.edu.cn}

\date{\today}

\begin{document}
	\maketitle
	\sloppy
	\begin{abstract}

	The paper defines a positive semidefinite operator called $\epsilon-$repelling Laplacian on a positive connected signed graph where $\epsilon$ is an arbitrary positive number less than a constant $\epsilon_0$ related to the graph's consensus problem. Then we investigate the upper bound of the second smallest eigenvalue of $\epsilon-$repelling Laplacian. Besides, we use the pseudoinverse of $\epsilon-$repelling Laplacian to construct a simplex as well as $\epsilon-$repelling cost whose square root turns out to be a distance among the vertices of the simplex. 
	
	We also extend the node and edge resistance curvature proposed by K.Devriendt et al. to node and edge $\epsilon-$repelling curvature and derive the corresponding Lichnerowicz inequalities on any positive connected signed graph. Moreover, it turns out that edge $\epsilon-$repelling curvature is no more than the Lin-Lu-Yau curvature of the underlying graph whose transport cost is $\epsilon-$repelling cost rather than the length of the shortest path.

	\end{abstract}
	
	\section{Introduction}
	A signed graph $\Gamma$ consists of an underlying weighted graph $G=(V,E)$ and a function $\sigma:E\rightarrow \{\pm 1\}$ satisfying $\sigma_{uv}=\sigma_{vu}$ for any $u\sim v$ that is usually called a connection or a signature on the signed graph. The earliest foundations of signed graphs can be traced back to Fritz Heider’s balance theory\cite{heider1946attitudes} in the 1940s, which classified triadic relationships (sets of three individuals) into "balanced" or "unbalanced". In his model, positive relationships (friendships) were denoted by +1 and negative relationships (hostility) by -1. Balanced configurations in Heider’s theory correspond to signed $3-$cycles which have a positive product of edge signs, thus establishing the basis for the study of signed graphs in sociology and psychology. In the 1950s, the concept of signed graphs was formally defined by Frank Harary\cite{cartwright1956structural}. F.Harary's structural balance theory provided a mathematical model for social dynamics, which states that a signed graph is balanced if and only if every cycle of the signed graph has a product of edge signs equal to $1$. Thomas Zaslavsky\cite{zaslavsky1982signed} provided a comprehensive formalization of the key concepts of signed graphs including the definitions of balance, switching equivalence, frustration index and the signed counterparts of classical graph matrices in the 1980s.  
	
	The signed graph is a useful tool for solving many problems in discrete geometry. \cite{bilu2006lifts} related signed graphs to $2-$lifts of graphs and reduced the problem of constructing expander graphs to finding a signature on a signed graph such that the signed graph has a small spectral radius. Positive edges on the signed graph represent "cross" connections on the covering graph, while negative edges represent "parallel" connections. In \cite{jiang2023spherical}, the authors established a novel and elegant connection between spherical two-distance sets and the spectral theory of signed graphs and thus proved a lower bound on the maximum size of a spherical two-distance set. Moreover, signed graphs also have applications in statistical physics, especially in the study of spin systems and the Ising model. The Ising model, initially developed by Ernst Ising\cite{ising1925beitrag}, is a mathematical model in which the vertices of a graph represent magnetic spins and the edges represent interactions between adjacent spins. The interactions can either be ferromagnetic (positive sign) or antiferromagnetic (negative sign).
	
	By the connection on a signed graph, the edge set can be partitioned into a positive edge set and a negative edge set. We call the graph consisting of the same vertices and positive(\text{resp.\ }negative) edge set as the positive(\text{resp.\ }negative) subgraph of the signed graph. We say a signed graph is positive connected(\text{resp.\ }negative) if its positive(\text{resp.\ }negative) subgraph is connected. As proposition \ref{swi} shows, any connected signed graph can be switching equivalent to one with the same underlying graph and connected positive subgraph. Thus the article pays attention to the positive connected signed graph.
	
	We first define a positive semidefinite operator $\Le$ with parameter $\epsilon$, called $\epsilon-$repelling Laplacian, where $L_{+}, L_{-}$ are Laplacians on the positive subgraph and negative subgraph of the signed graph respectively. We discuss the range of $\epsilon$ in detail in the following sections. We derive some upper bounds of the second smallest eigenvalue of $\Le$ in Theorem \ref{eigdeg} and Theorem \ref{main2}. As you can see, when $\epsilon=-1$, $\Le$ is the Laplacian on the underlying graph $G$. So our theorems can recover existing results on Laplacian second smallest eigenvalue on graphs. 
	
	\begin{theorem}\label{eigdeg}
		Suppose $(G,\sigma,w)$ is a non-complete positive-connected signed graph and $\epsilon_0$ is its consensus index. For $\forall \epsilon<\epsilon_0$, let $\lambda_2^{(\epsilon)}$ be the second smallest eigenvalue of $\Le$ and $M_{\epsilon}:=\max_{x\in V}d_x^{+}-\epsilon d_x^{-}$. Then we have
		\[ \lambda_2^{(\epsilon)}\le M_{\epsilon} \]
	\end{theorem}
	
	\begin{theorem}\label{main2}
		Suppose $(G,\sigma,w)$ is a positive-connected and negative-connected signed graph with $n$ vertices and $\epsilon_0$ is its consensus index. For $\forall \epsilon<\epsilon_0$, let $\lambda_2^{(\epsilon)}$ be the second smallest eigenvalue of $\Le$. Denote by $D$ the diameter of $G$, by $d_{\text{max}}^{+}$ the maximum of $d^{+}$ on $V_{+}$ and by $\mu_{-}^0$ the minimum of $w^{-}$ on $E_{-}$. Then
		\[ \lambda_2^{(\epsilon)}\le 2d_{\text{max}}^{+}-\frac{\epsilon\mu_{-}^0}{Dn} \]
	\end{theorem}
		
	Then we use the pseudoinverse of $\epsilon-$repelling laplacian to define a kind of "resistance" on the positive connected signed graph, called $\epsilon-$repelling cost and denoted by $\Om(\cdot,\cdot)$. In Theorem \ref{main3}, we demonstrate the monotonicity of $\Om(i,j)$ with respect to $\epsilon$ and thus know the $\epsilon-$repelling cost between any two vertices is no less than resistance between them if $\epsilon>0$.
	\begin{theorem}\label{main3}
		For any $\epsilon_1<\epsilon_2<\epsilon_0$ and for $\forall i,j \in V$:
		\[ \Omega_{\epsilon_1}^{\sigma}(i,j)\le \Omega_{\epsilon_2}^{\sigma}(i,j) \]	
		Especially, let $\Omega(i,j)$ be the resistance on the underlying graph $G$, then for any $0<\epsilon<\epsilon_0$:
		\[ \Omega(i,j)\le \Omega_{\epsilon}^{\sigma}(i,j) \]	
	\end{theorem}
	
	The Lichnerowicz Spectral Inequality traces its origins to the foundational work of André Lichnerowicz. His pioneering efforts\cite{lichnerowicz1958book} marked the beginning of a deeper understanding of the interplay between the spectrum of the Laplace–Beltrami operator and the geometric curvature of Riemannian manifolds. Specifically, he proved that for a compact 
	$n$-dimensional Riemannian manifold $(M,g)$ with Ricci curvature bounded below by $Ric\ge (n-1)K$ for some $K>0$, the first non-zero eigenvalue $\lambda$ of the Laplacian satisfies:
	\[ \lambda\ge nK \]
	In the discrete settings, Y.Lin et al.\cite{lin2011ricci} introduced a notion of Lin-Lu-Yau curvature on graphs and established a discrete analogue of the Lichnerowicz inequality. In \cite{liu2019curvature}, the authors considered curvature dimension inequalities with signatures and proved the Lichnerowicz inequality on the connection graphs. S.Steinerberger\cite{Steinerberger2017} proposed a new approach to defining curvature on graphs using equilibrium measures and proved the corresponding discrete Lichnerowicz inequality while the form of Lichnerowicz inequality in \cite{devriendt2024graph} is its special case.  
	
	Based on $\Om(\cdot,\cdot)$ we define two new curvatures on the vertex set $V$ and the product of vertex set $V\times V$ of positive connected signed graphs, called node $\epsilon-$repelling curvature $\tau_{\epsilon}(\cdot)$ and edge $\epsilon-$repelling curvature $\vartheta_{\epsilon}(\cdot,\cdot)$, generalizing the preceding definitions of node and edge resistance curvature proposed in \cite{devriendt2022effective}. 
	
	We obtain the version of Lichnerowicz inequality in terms of node $\epsilon-$repelling curvature: the second smallest eigenvalue of $\epsilon-$repelling Laplacian has a lower bound related to $K_{\epsilon}$ if $\tau_{\epsilon}$ is bounded below by a positive constant $K_{\epsilon}$ in Theorem \ref{main4}, as a generalization of \cite[Theorem 2]{devriendt2024graph}.
	
	\begin{theorem}\label{main4}
		Suppose node $\epsilon-$repelling curvature $\tau_{\epsilon}$ on positive connected signed graph $(G,\sigma,w)$ has positive lower bound $K_{\epsilon}$, then the second smallest eigenvalue $ \lambda_2^{(\epsilon)}$ of $\Le$ satisfies:
		\[ \lambda_2^{(\epsilon)}\ge \frac{2K_{\epsilon}}{|V(G)|} \]
	\end{theorem}

	Furthermore, we demonstrate the edge $\epsilon-$repelling curvature is no more than the Lin-Lu-Yau curvature using $\epsilon-$repelling cost as transport cost in Theorem \ref{thmv}. 
	  \begin{theorem}\label{thmv}
	 	For any $(i,j)\in E(G)$ we have
	 	\[ \vartheta_{\epsilon}(i,j)\le \kappa^{(LLY)}(i,j) \]
	 \end{theorem}
	 
	Besides, we also derive a kind of Lichnerowicz inequality in terms of edge $\epsilon-$repelling curvature. We obtain one lower bound of the second smallest eigenvalue of the Laplacian on underlying graph in terms of the positive lower bound of edge $\epsilon-$repelling curvature in Theorem \ref{main5}, which can be regarded as an improved version of \cite[Theorem 4.2]{lin2011ricci}.
	\begin{theorem}\label{main5}
		Assume $\vartheta_{\epsilon}(i,j)\ge k_{\epsilon}> 0,\, i,j\in V $ on $(G,\sigma,w)$ and $\mu_2$ is the second smallest eigenvalue of the Laplacian $Q$ on underlying graph $G$, then
		\[ \mu_2\ge k_{\epsilon} \]
	\end{theorem}
 	In the end, we compute the node and edge $\epsilon-$repelling curvatures on some positive connected signed graphs which share the same underlying graph structure but possess non-isomorphic connections.
	\section{$\epsilon-$repelling Laplacian on positive connected signed graphs}
	 Assume $\Gamma=(G,\sigma,w)$ is a connected signed graph where $V$ is the vertex set, $E$ is the edge set, $\sigma : E \to \{\pm1\}$ is one-dimensional, and $w : E \to \mathbb{R}_{\geq 0}$ is the graph’s weight. Denote by $E_{+}$ the set of edges in $E$ with positive sign $+1$. That is, $E_{+} = \sigma^{-1}(+1)$.	Similarly, denote by $E_{-}$ the set of edges with negative sign $-1$. For a vertex $i$, let $d_i := \sum_{j\sim i} w_{ij}$ be the degree of $i$, $d^+_i := \sum_{(i,j) \in E^+} w_{ij}$
	be the positive degree, and $d^-_i := \sum_{(i,j) \in E^-} w_{ij}$
	be the negative degree. Define $w^{+}, w^{-}$ as the restriction of $w$ on $E_{+}$ and $E_{-}$ respectively and extend them by zero to $E$. We call $(V, E_{+}, w^{+})$ the positive subgraph and $(V, E_{-}, w^{-})$ the negative subgraph of the signed graph $(G,\sigma,w)$. 
	
	Define two Laplacians on the positive subgraph and negative subgraph as follows:
	\[ 
		L_{+}(i,j)=
		\begin{cases}
			\dz &\text{if}\, i=j\\
			-w_{ij}&\text{if}\, (i,j)\in E_{+}\\
			0&\text{otherwise}
		\end{cases}
		\quad
		L_{-}(i,j)=
		\begin{cases}
			\df &\text{if}\, i=j\\
			-w_{ij}&\text{if}\, (i,j)\in E_{-}\\
			0&\text{otherwise}
		\end{cases}
	 \]
	Shi et al. \cite{shi2019dynamics} used these two Laplacians to propose the repelling negative dynamic for signed networks. The form is as follows:  
	\[
	X(t +1) = (I - \alpha L_+ - \beta L_-) X(t).
	\]
	They call $\alpha L_+ - \beta L_-$ the repelling weighted Laplacian of the signed graph. We study $L_+ - \epsilon L_-$, differing from their repelling Laplacian by a constant factor, called the $\epsilon$-repelling Laplacian in our paper.  
	
	We always assume the signed graph $\Gamma$ is positive connected, unless stated otherwise. The reason of the assumption is that every connection on connected signed graph is switching equivalent to a connection whose positive set is connected. The following proposition illustrate that any signed graph is switching equivalent to a positive connected signed graph.  And readers can refer to \cite[Lemma2.5]{cloninger2024random} for the similar proposition on connection graph.
	
	\begin{proposition}\label{swi}
		Assume $(G,\sigma,w)$ is a signed graph with connected underlying graph, then for any spanning tree $T$, there exists a connection $\theta^T$ such that it is switching equivalent to $\sigma$ and it takes positive sign on $E(T)$.
	\end{proposition}
	
	\begin{proof}
		Fix a vertex $i_0$, for $\forall i\in V$ there exists a unique path $P_{i_0i}$ on T from $i_0$ to $i$.
		Define $\theta_T$ as follows:
		\[ \theta^T_{ij}:=\sigma_{P_{i_0i}}\sigma_{ij}\sigma_{P_{i_0j}} \]
		That is, $f(i):=\sigma_{P_{i_0i}}$ is the switching function. It's easy to check that $\theta^T_{ij}=1$ if $(i,j)\in E(T)$.
	\end{proof}

	\subsection{The consensus index $\epsilon_0$}
	As is well known, the Laplacian operator on the graph without connection is crucial and useful in studying the graph’s geometry. So we pay more attention to those $\epsilon$ such that the $\epsilon$-repelling Laplacian has nice properties similar to the original Laplacian. 
	
	The following lemma tells what $\epsilon$ we care about.
	
	\begin{lemma}\label{lem1}
		If $(G,\sigma,w)$ is a positive connected signed graph, then there exists $\epsilon > 0$ such that $L_+ - \epsilon L_-$ is a positive semidefinite matrix with rank $n - 1$.  
	\end{lemma} 
	
	\begin{proof}
		Note that $(L_+ - \epsilon L_-) \vec{1} = 0$, so $0$ is an eigenvalue of $L_+ - \epsilon L_-$ for any $\epsilon > 0$. We only need to prove that there exists $\epsilon > 0$ such that the second smallest eigenvalue of $L_+ - \epsilon L_-$ is positive.
		
		Assume $\{\lambda_i(A)\}_{i=1}^{n}$ are the non-decreasing eigenvalues of any $n \times n$ matrix $A$. From Weyl’s inequalities \cite{bhatia2013matrix}, which show the relations between the eigenvalues of two Hermitian matrices and their sum, we know  
		\begin{equation}\label{equ1}
			\lambda_2(L_+ - \epsilon L_-) \geq \lambda_2(L_+) - \epsilon \lambda_n(L_-).
		\end{equation}
		Since the positive subgraph is connected, $\lambda_2(L_+) > 0$. Therefore, we can take $\epsilon$ as any real number less than $\frac{\lambda_2(L_+)}{\lambda_n(L_-)}$ such that $\lambda_2(L_+ - \epsilon L_-) > 0$.
	\end{proof}
	
	\begin{remark}\label{rmk1}
		In fact, the inequality (\ref{equ1}) holds with equality if and only if $\lambda_2(L_+)$ and  $\lambda_n(L_-)$ have common eigenvectors, which cannot be satisfied in general, meaning we can take $\epsilon$ as $\frac{\lambda_2(L_+)}{\lambda_n(L_-)}$ in Lemma (\ref{lem1}). Besides, the condition that the signed graph is positive connected is indispensable.		
	\end{remark}

	\begin{definition}
		We define $\epsilon_0$ as the maximum value of $\epsilon$ satisfying Lemma (\ref{lem1}). That is,  
		\[
		\epsilon_0 := \sup \{ \epsilon > 0 \mid L_+ - \epsilon L_- \text{ is semi-positive definite with rank } n - 1 \}.
		\]	
		We call $\epsilon_0$ the consensus index of the positive-connected signed graph.	
	\end{definition}
	
	\begin{remark}
		The reason why we call $\epsilon_0$ the consensus index is that it's a critical value at which the second smallest eigenvalue of repelling Laplacian $L_{+}-\epsilon L_{-}$ is nonzero and $\lambda_2(L_{+}-\epsilon L_{-})$ can be a measure of speed of consensus algorithm\cite{olfati2004consensus}. 
	\end{remark}
	
	The signed graph $(G,\sigma,w)$ can be regarded as the graph $(G,\tw)$ with negative weights where 
	\[ \tw:=w^{+}-\epsilon w^{-} \]
	Then the $\epsilon-$repelling Laplacian $\Le$ on $(G,\sigma,w)$ is the original Laplacian $\LE$ on $(G,\tw)$.  The form of $\LE$ is as follows:
	\[ 
	\LE(i,j)=
	\begin{cases}
		\dz-\epsilon\df&\text{if}\, i=j\\
		-w_{ij}^{+}&\text{if}\, (i,j)\in E_{+}\\
		\epsilon w_{ij}^{-}&\text{if}\, (i,j)\in E_{-}\\
		0&\text{otherwise}
	\end{cases}
	 \]
	 The research\cite{chen2016definiteness,zelazo2014definiteness} about the graph with negative weight require an assumption that any two distinct edges with negative weights are not contained in a cycle. 
	 For convenience, we call this assumption "no negative cycle" assumption.
	 
	 \begin{theorem}
	 	Suppose $(G,\sigma,w)$ is a positive connected signed graph satisfying "no negative cycle" assumption and $\epsilon_0$ is its consensus index. Then for $\forall (i,j)\in E_{-}$, we have 
	 	\[ \epsilon_0 \le \frac{1}{w_{ij}r_{ij}} \]
	 	where $r_{ij}$ is the resistance distance in the positive subgraph $(V, E_{+}, w^{+})$.
	 \end{theorem}  
	 
	 \begin{proof}
	 	For $\forall \epsilon<\epsilon_0$, we know $\Le$ is positive semidefinite. As a result of \cite[Thm3.2]{chen2016definiteness}, we have $\epsilon w_{ij}\le \frac{1}{r_{ij}}$ for any $(i,j)\in E_{-}$. Therefore, $\epsilon\le \frac{1}{w_{ij}r_{ij}}$ and let $\epsilon\rightarrow \epsilon_0$, we finish the proof.
	 \end{proof}
	 
	The balance of the signed graph was introduced by Frank Harary\cite{harary1953notion}. By the definition of balance, the vertex set $V$ of a balanced signed graph can be partitioned into two complementary sets $V_1$ and $V_1^C$, called Harary bipartition, such that the endpoints of every edge with positive sign are either in $V_1$ or $V_1^C$. So a balanced signed graph must not be positive-connected if both $E_{+}(V_1)$ and $E_{+}(V_2)$ are not empty. On the other hand, researchers tend to use connection Laplacian\cite{chung2014ranking} to study the balance of signed graph. Note the consensus index is defined under the premise that the signed graph is positive-connected. A natural question is what relation between the balance and consensus index is. The following lemma tells us that the consensus index doesn't exist on signed graphs which are not only positive-connected but also balanced.
	
	\begin{lemma}
		If $(G,\sigma,w)$ is a not only positive-connected but also balanced signed graph, for $\epsilon>0$, the $\epsilon-$repelling Laplacian $\Le$ is not positive semidefinite.
	\end{lemma}
	
	\begin{proof}
		Assume $A, A^C\subset V$ is a Harary bipartition of $(G,\sigma,w)$. For any $a\neq 1$, define $f:V\rightarrow \mathbb{R}$ as follows:
		\[ 
		f(i)=
		\begin{cases}
			1&\text{if}\, i\in A\\
			a&\text{if}\, i\in A^C
		\end{cases}
		 \]
		 Then we have
		 \[ f^T\left(\Le\right)f=-\epsilon\sum_{(i,j)\in E_{-}}w^{-}_{ij}(1-a)^2<0 \]		 
	\end{proof}
	
	\subsection{The spectra gap of $\epsilon-$repelling Laplacian }
	In this section, we prove Theorem \ref{eigdeg} and Theorem \ref{main2}. Recall Theorem \ref{eigdeg}, it derives the upper bound of the second smallest eigenvalue of $\epsilon-$repelling Laplacian in terms of the positive degree, negative degree and $\epsilon$.

	\begin{proof}[Proof of Theorem \ref{eigdeg}]
		Since $G$ is non-complete, there exist two distinct vertices $x$ and $y$ such that $x\nsim y$.
		Define $f:V\rightarrow \mathbb{R}$ as follows:
		\[ f:=\mathbf{1}_{x}-\mathbf{1}_{y} \]
		where $\mathbf{1}_{x},\mathbf{1}_{y} $ are the characteristic functions of $x, y$.
		Then we have 
		\begin{align*}
			\mathbf{1}_{x}^T(\Le)\mathbf{1}_{x}&=\frac{1}{2}\sum_{z,h\in V}(w_{zh}^{+}-\epsilon w_{zh}^{-})(\mathbf{1}_{x}(z)-\mathbf{1}_{x}(h))^2\\
			&=\frac{1}{2}(\sum_{z=x,h\neq x}+\sum_{z\neq x,h=x})(w_{zh}^{+}-\epsilon w_{zh}^{-})(\mathbf{1}_{x}(z)-\mathbf{1}_{x}(h))^2\\
			&=\sum_{z=x,h\neq x} (w_{zh}^{+}-\epsilon w_{zh}^{-})(\mathbf{1}_{x}(z)-\mathbf{1}_{x}(h))^2\\
			&=\sum_{h\neq x}(w_{xh}^{+}-\epsilon w_{xh}^{-})\\
			&=d_x^{+}-\epsilon d_x^{-}			
		\end{align*}
		Similarly, $\mathbf{1}_{y}^T(\Le)\mathbf{1}_{y}=d_y^{+}-\epsilon d_y^{-} $.
		
		It is obvious that the product function of $\mathbf{1}_{x},\mathbf{1}_{y}$ is zero. We claim that the product function of $(\Le)\mathbf{1}_{x},\mathbf{1}_{y}$ is zero. In fact, if $z\neq y$, $\mathbf{1}_{y}(z)=0$. If $z=y$, $(\Le)\mathbf{1}_{x}(z)=\mathbf{1}_{x}(y)-\sum_{h\sim y}(w_{yh}^{+}-\epsilon w_{yh}^{-})(\mathbf{1}_{x}(y)-\mathbf{1}_{x}(h))^2=0 $ since $x\nsim y$ and $h$ can not take $x$. 
		Similarly, the product of $(\Le)\mathbf{1}_{y},\mathbf{1}_{x}$ is also zero. 
		Therefore, \[ f^T(\Le)f=\mathbf{1}_{x}^T(\Le)\mathbf{1}_{x}+\mathbf{1}_{y}^T(\Le)\mathbf{1}_{y}=d_x^{+}-\epsilon d_x^{-}+d_y^{+}-\epsilon d_y^{-} \]
		\[ f^Tf=\mathbf{1}_{x}^T\mathbf{1}_{x}+\mathbf{1}_{y}^T\mathbf{1}_{y}=2\]
		So the Raleigh quotient of $\Le$ is 
		\begin{align*}
			\mathcal{R}(f)&=\frac{f^T(\Le)f}{f^Tf}\\
			&=\frac{d_x^{+}-\epsilon d_x^{-}+d_y^{+}-\epsilon d_y^{-}}{2}\\
			&\le M_{\epsilon}
		\end{align*}
		Since $f\perp \vec{1}$ and $\lambda_2^{(\epsilon)}=\inf\limits_{g\perp \vec{1}}\mathcal{R}(g)$, we have 
		\[ \lambda_2^{(\epsilon)}\le \mathcal{R}(f) \le M_{\epsilon}\]
	\end{proof}
	
	Next, we prove Theorem \ref{main2}, which gives a upper bound of $\lambda_2^{(\epsilon)}$ in the case where the signed graph is not only positive-connected but also negative-connected.

	\begin{proof}[Proof of Theorem \ref{main2}]
		Assume $f$ is an eigenfunction of $\lambda_2^{(\epsilon)}$ with $\sum_{x\in V}f^2(x)=1$, then $f\perp \vec{1}$ and $\lambda_2^{(\epsilon)}=\mathcal{R}(f)$. Assume $\max_{x\in V}f(x)=\max_{x\in V}|f(x)|$, otherwise we replace $f$ by $-f$. According to the pigeonhole principle, there exists a vertex $x_0$ such that $f(x_0)=\max_{x\in V}f(x)\ge \frac{1}{\sqrt{n}}$. From $f\perp \vec{1}$ we know there exists a vertex $y_0$ such that $f(y_0)<0$. Since $(G,\sigma,w)$ is negative-connected, let $x_0=z_1\sim z_2\sim \cdots\sim z_m=y_0$ be a path consisting of edges with negative sign.
		The Rayleigh quotient $\mathcal{R}(f) $ can split into two parts:
		\begin{align*}
			\mathcal{R}(f)&=\sum_{(i,j) \in E^+}w_{ij}^{+}(f(i)-f(j))^2-\epsilon \sum_{(i,j) \in E^-}w_{ij}^{-}(f(i)-f(j))^2\\
			&:=A-\epsilon B
		\end{align*}
	Next we estimate $A$ as follows:
	\begin{align*}
		A&\le 2\sum_{(i,j) \in E^+}w_{ij}^{+}(f^2(i)+f^2(j))=\sum_{i,j\in V}w_{ij}^{+}(f^2(i)+f^2(j))\\
		&=2\sum_{i,j\in V}w_{ij}^{+}f^2(i)=2\sum_{i\in V}d_{i}^{+}f^2(i)\le 2d_{\text{max}}^{+}
	\end{align*}
	Then we derive a lower bound for $B$:
	\begin{align*}
		B&\ge \mu_{-}^0\sum_{(i,j) \in E^-}(f(i)-f(j))^2\ge \mu_{-}^0\sum_{k=1}^{m-1}(f(z_k)-f(z_{k+1}))^2\\
		&\ge \frac{\mu_{-}^0}{m}\left(\sum_{k=1}^{m-1}f(z_k)-f(z_{k+1})\right)^2= \frac{\mu_{-}^0}{m}(f(x_0)-f(y_0))^2\ge \frac{\mu_{-}^0}{Dn}
	\end{align*}
	Therefore, \[  \lambda_2^{(\epsilon)}=\mathcal{R}(f)\le 2d_{\text{max}}^{+}-\frac{\epsilon\mu_{-}^0}{Dn} \]
	\end{proof}
	\subsection{The simplex of $\epsilon-$repelling Laplacian}
	 In this section, we always assume $(G,\sigma,w)$ is positive-connected with $n+1$ vertices, the underlying graph $G$ is connected and $\epsilon$ is a positive number less than the signed graph's consensus index $\epsilon_0$. 
	 
	 Compared with K.Devriendt\cite{devriendt2022effective}'s work which corresponded the graph Laplacian to a hyperacute simplex for proving effective resistance is a distance on graph, we establish a correspondence  between the $\epsilon-$repelling Laplacian on $(G,\sigma,w)$ and $n-$simplex and define a new metric on signed graph. Our work is inspired by \cite[Chapter 3]{fiedler2011matrices} about the signed graph of the simplex, which uses the sign on graph to distinguish whether the dihedral angles of simplex are obtuse or acute.
	 
	 Since $\epsilon-$repelling Laplacian satisfies the spectrum property $(i)_{\sigma}-(iii)_{\sigma}$ in \cite{devriendt2022effective}, the correspondence stated in following lemma is right as a corollary of \cite[Proposition 2]{devriendt2022effective}.
	\begin{lemma}
		There exists some $n-$simplex with vertex matrix $S=\left[v_1,\cdots,v_{n+1}\right]^T$ whose centroid is the original point in $\mathbb{R}^n$ such that $\epsilon-$repelling Laplacian is the pseudoinverse of its canonical Gram matrix. That is, 
		\[ \Le=(SS^T)^{\dagger} \]
		We call this simplex $\mathcal{S}_{\epsilon}$ the simplex of $\epsilon-$repelling Laplacian.
	\end{lemma}
	
	We denote by $\mathcal{S}$  the hyperacute simplex of Laplacian on the underlying graph of signed graph in K.Devriendt\cite{devriendt2022effective}'s work.  Compare $\mathcal{S}_{\epsilon}$ with $\mathcal{S}$, we can obtain some information about the signed graph. First of all, the dihedral angles of $\mathcal{S}_{\epsilon} $ are different from $\mathcal{S}$ whose dihedral angles are all acute.  Denote by $\phi_{ij}^{\sigma}$ the dihedral angle between the sub-simplices $\mathcal{S}_{\epsilon}\setminus \{v_i\} $ and $\mathcal{S}_{\epsilon}\setminus \{v_i\}$. Then $\phi_{ij}^{\sigma}>\frac{\pi}{2} $  if and only if $\sigma_{ij}=-1$. Besides, the following lemma shows that we can identify the vertices that are not connected to negative edges from the difference of altitudes of $\mathcal{S}_{\epsilon}$ and $\epsilon$.
	
	\begin{lemma}
		For any $i=1,2,\cdots,n+1$, let $l_i(\mathcal{S}_{\epsilon}), l_i(\mathcal{S})$ be the lengths of the altitudes from $v_i$ in $\mathcal{S}_{\epsilon} $ and $\mathcal{S}$ respectively. Then 
		\[ l_i(\mathcal{S}_{\epsilon})\ge l_i(\mathcal{S}) \]
		where the equality holds if and only if all the edges connected to the node $i$ are positive.
	\end{lemma}
	
	\begin{proof}
		The Gramian of $\mathcal{S}_{\epsilon}$ is $\Le$ and the Gramian of $\mathcal{S}$ is $L_{+}+L_{-}$. Assume $\Le=\left[a_{ij}\right]$ and $L_{+}+L_{-}=\left[b_{ij}\right] $. Then $a_{ii}\le b_{ii}$ for any $i=1,\cdots,n+1$ and the equality holds if and only if $d_i^{-}=0$. From \cite[Thm2.1.1]{fiedler2011matrices}, we know 
		\[ l_i(\mathcal{S}_{\epsilon})=\frac{1}{\sqrt{a_{ii}}}\ge \frac{1}{\sqrt{b_{ii}}}=l_i(\mathcal{S}) \]
	\end{proof}

	\begin{example} Consider a positive-connected signed $3-$cycle in Figure \ref{fig1}, then 
		
			\[
			L_{+} = \begin{pmatrix}
				2 & -1 & -1 \\
				-1 & 1 & 0 \\
				-1 & 0 & 1
			\end{pmatrix}
			\quad
			\text{Spec}(L_{+})=\{3,1,0\}
			\]

			\[
			L_{-} = \begin{pmatrix}
				0 &  &  \\
				 & 1 & -1 \\
				 & -1 & 1
			\end{pmatrix}\quad
			\text{Spec}(L_{-})=\{2,0,0\}
			\]
		Then $\frac{\lambda_2(L_{+})}{\lambda_3(L_{-})}=\frac{1}{2}$ and we take $\epsilon$ as $\frac{1}{4}$. We can compute the pesudoinverse of $\Le$ as follows and the vertex matrix of the corresponding simplex:
		\begin{align*}
			(\Le)^{\dagger}&=
			\begin{pmatrix}
				0.222&-0.111&-0.111\\
				-0.111&1.055&-0.944\\
				-0.111&-0.944&1.055
			\end{pmatrix}\\
			&=\begin{pmatrix}
				0&0.471\\
				-1&-0.236\\
				1&-0.236
			\end{pmatrix}
			\begin{pmatrix}
				0&-1&1\\
				0.471&-0.236&-0.236
			\end{pmatrix}
		\end{align*}	
     So the simplex of $\frac{1}{4}-$repelling Laplacian is a triangle in Figure \ref{fig2}. You can see the sign on the edge $\overline{v_2v_3}$ is $-1$ and the dihedral angle between the sub-simplex $\overline{v_1v_3} $ and $\overline{v_1v_2}$ is obtuse.
	\end{example}
	 
	\begin{figure}[htbp]
		\centering
		\begin{minipage}{0.4\textwidth}
			\centering
			\begin{tikzpicture}
				\draw[thick] (0,0) -- (2,0) node[midway, below] {+1} 
				-- (1,1.732) node[midway, right] {-1}
				-- (0,0) node[midway, left] {+1} 
				-- cycle;
				
				\node at (0, 0) [below left] {1};
				\node at (2, 0) [below right] {2};
				\node at (1, 1.732) [above] {3};
			\end{tikzpicture}
			\caption{signed $3-$cycle}
			\label{fig1}
		\end{minipage}
		\hspace{0.04\textwidth}
		\begin{minipage}{0.4\textwidth}
			\centering
			\begin{tikzpicture}
				\draw[thick] (0,0.471) -- (-1,-0.236) -- (1,-0.236) -- cycle;
				\node at (0, 0.471) [above] {$(0, 0.471)$};
				\node at (0, 0.471) [below] {$v_1$};
				\node at (-1, -0.236) [below left] {$(-1, -0.236)$};
				\node at (-1, -0.236) [below right] {$v_2$};
				\node at (1, -0.236) [below right] {$(1, -0.236)$};	
				\node at (1, -0.236) [below left] {$v_3$};			
			\end{tikzpicture}
			\caption{The corresponding $3-$simplex}
			\label{fig2}
		\end{minipage}		
	\end{figure}
	
	\begin{definition}
		Define a matrix $\Omega^{\sigma}_{\epsilon}\in \mathbb{R}^{(n+1)\times(n+1)}$, called $\epsilon-$repelling matrix, as follows:
		\[ \Omega^{\sigma}_{\epsilon}(i,j):=(\vec{e_i}-\vec{e_j})^TL_{\epsilon}^{\dagger}(\vec{e_i}-\vec{e_j}) \] 
		We call $\Om(i,j)$ the $\epsilon-$repelling cost between $i$ and $j$.
		
	\end{definition}
	 Denote by $L_{\epsilon}^{\dagger}$ the pseudoinverse of $\Le$ and from the above theorem we know $L_{\epsilon}^{\dagger}=SS^T$ where $S$ is the vertex matrix of some $n-$simplex. 
	 \begin{proposition}
	 	$\sqrt{\Omega^{\sigma}_{\epsilon}(i,j)}$ is a distance on signed graph $\Gamma$.
	 \end{proposition}
	 \begin{proof}
	 	Assume $L_{\epsilon}^{\dagger}=SS^T$ where $S=\left[v_1,\cdots,v_{n+1}\right]^T$ is the vertex matrix of some $n-$simplex,
	 	\[ \Omega^{\sigma}_{\epsilon}(i,j)=\left(v_i-v_j\right)^T\left(v_i-v_j\right)=||v_i-v_j ||^2\ge 0\]
	 	\[ \sqrt{\Omega^{\sigma}_{\epsilon}(i,j)}=||v_i-v_j || \]
	 	Therefore, for $\forall i,j,k\in V$, 
	 	\[ \sqrt{\Omega^{\sigma}_{\epsilon}(i,j)}=||v_i-v_j ||\le||v_i-v_k ||+||v_k-v_j ||=\sqrt{\Omega^{\sigma}_{\epsilon}(i,k)}+\sqrt{\Omega^{\sigma}_{\epsilon}(j,k)} \]
	 \end{proof}
	 
	 \begin{remark}
	 	From \cite[Property 1]{devriendt2022effective}, we know the off-diagonal entries of pseudoinverse of canonical Gram matrix reflect whether the dihedral angles of simplex are obtuse, right or acute. Because the off-diagonal entries of $\Le$ are both positive and negative, the simplex of $\epsilon-$repelling Laplacian is not a hyperacute simplex and the angles of its triangular face may be obtuse, which results in $\Omega^{\sigma}_{\epsilon}(\cdot,\cdot) $ not being a distance.
	 \end{remark}
	 
	 \begin{lemma}
	 The signed resistance matrix $\Omega^{\sigma}_{\epsilon}$ is invertible.
	 \end{lemma}
	 
	 \begin{proof}
	 	Denote by $\zeta$ the vector consisting of the diagonal entries of $L_{\epsilon}^{\dagger}$. From definition, we have 
	 	\[ \Omega^{\sigma}_{\epsilon}=\zeta\vec{1}^T+\vec{1}\zeta^T-2L_{\epsilon}^{\dagger} \]
	 	$L_{\epsilon}^{\dagger}$ is positive semidefinite whose null space is spanned by all ones vector since it's the pesudoinverse of $\Le$. Then for any $\vec{x}\perp \vec{1}$:
	 	\[ \vec{x}^T\Omega^{\sigma}\vec{x}=-2\vec{x}^TL_{\epsilon}^{\dagger}\vec{x}<0 \]
	 	For $\vec{x}=c\vec{1}$ where $c$ is a constant:  
	 	\[ \vec{x}^T\Omega^{\sigma}_{\epsilon}\vec{x}=c^2\vec{1}^T\Omega^{\sigma}\vec{1}>0 \]
	 	Therefore, for any $\vec{x}\neq \vec{0}$, $\vec{x}^T\Omega^{\sigma}_{\epsilon}\vec{x}\neq0 $, implying that $\Omega^{\sigma}_{\epsilon}$ is invertible.
	 \end{proof}
	
	 From \cite[Equation (14)]{devriendt2022effective} and \cite[Thm2.1.1]{fiedler2011matrices}, we obtain a matrix identity on positive-connected signed graph in the following theorem.
	 
	 \begin{theorem}\label{eq2}
	 	\[ 
	 	-\frac{1}{2}\begin{pmatrix}
	 		0&\vec{1}^T\\
	 		\vec{1}&\Omega^{\sigma}_{\epsilon}
	 	\end{pmatrix}
	 	=\begin{pmatrix}
	 		4R^2&-2\vec{r}^T\\
	 		-2\vec{r}^T& \Le
	 	\end{pmatrix}^{-1}	 	
	 	 \]
	 	 where $R$ is the radius of the circumscribed hypersphere of the simplex of $\Le$ and $\vec{r}$ is barycentric coordinate of the circumcenter of the simplex of $\Le$.
	 \end{theorem}
	 
	  Next we compare the original resistance $\Omega(i,j)$ on the underlying graph $G$ with the signed resistance $\Omega^{\sigma}_{\epsilon}(i,j)$ on $\Gamma$. Recall the definition of resistance distance $\Omega(i,j)$ on general graphs 
	  \[ \Omega(i,j):=(\vec{e_i}-\vec{e_j})^TL^{\dagger}(\vec{e_i}-\vec{e_j}) \]
	  where $L=L_{+}+L_{-}$ is the Laplacian on the underlying graph $G$. First we give a proposition about the expression of the pesudoinverse of any symmetric matrix in terms of the null space.
	  
	  \begin{proposition}\label{pro1}
	  	Suppose $A$ is a symmetric and singular matrix and $\{\vec{z_1},\cdots,\vec{z_m}\}$ are the orthonormal basis of the null space of $A$, then the matrix 
	  	\[ A_0:=A+\sum_{j=1}^{m}\vec{z_j}\vec{z_j}^T \]
	  	 is invertible and the inverse is 
	  	 \[ A_0^{-1}=A^{\dagger}+\sum_{j=1}^{m}\vec{z_j}\vec{z_j}^T \]
	  \end{proposition}
	 
	 \begin{proof}
	 	\begin{align*}
	 		&\left(A+\sum_{j=1}^{m}\vec{z_j}\vec{z_j}^T\right)\left(A^{\dagger}+\sum_{j=1}^{m}\vec{z_j}\vec{z_j}^T\right)\\
	 		=&AA^{\dagger}+\sum_{j=1}^{m}A\vec{z_j}\vec{z_j}^T+\sum_{j=1}^{m}\vec{z_j}\vec{z_j}^TA+\sum_{i=1}^{m}\sum_{j=1}^{m}\vec{z_i}\vec{z_i}^T\vec{z_j}\vec{z_j}^T\\
	 		=&AA^{\dagger}+\sum_{j=1}^{m}\vec{z_j}\vec{z_j}^T\\
	 		=&I
	 	\end{align*}
	 	Similarly we have 
	 	\[ \left(A^{\dagger}+\sum_{j=1}^{m}\vec{z_j}\vec{z_j}^T\right)\left(A+\sum_{j=1}^{m}\vec{z_j}\vec{z_j}^T\right)=I \]
	 	Therefore, we finish the proof.
	 \end{proof}
	 
	 The next theorem shows $\epsilon-$repelling cost on signed graph is non-decreasing with respect to $\epsilon$  in case where $\epsilon<\epsilon_0$ may be negative. Note $\Le$ becomes the Laplacian on underlying graph $G$ when $\epsilon=-1$, hence we can compare the magnitudes of $\Omega_{\epsilon}^{\sigma}(i,j) $ and $\Omega(i,j) $.

	 \begin{proof}[Proof of Theorem \ref{main3}]
	 	Since both the orthonormal bases of null spaces $N(L_{+}-\epsilon_1L_{-})$ and $N(L_{+}-\epsilon_2L_{-})$ are $\{\frac{1}{\sqrt{n+1}}\vec{1}\}$, using Proposition (\ref{pro1}), we have
	 	\begin{align*}
	 		(L_{+}-\epsilon_1L_{-})^{\dagger}&=\left(L_{+}-\epsilon_1L_{-}+\frac{1}{n+1}\vec{1}\vec{1}^T\right)^{-1}-\frac{1}{n+1}\vec{1}\vec{1}^T\\
	 		(L_{+}-\epsilon_2L_{-})^{\dagger}&=\left(L_{+}-\epsilon_2L_{-}+\frac{1}{n+1}\vec{1}\vec{1}^T\right)^{-1}-\frac{1}{n+1}\vec{1}\vec{1}^T
	 	\end{align*}
	 	 Then
	 	 \begin{align*}
	 	 	&(L_{+}-\epsilon_2L_{-})^{\dagger}-(L_{+}-\epsilon_1L_{-})^{\dagger}\\
	 	 	=&\left(L_{+}-\epsilon_2L_{-}+\frac{1}{n+1}\vec{1}\vec{1}^T\right)^{-1}-\left(L_{+}-\epsilon_1L_{-}+\frac{1}{n+1}\vec{1}\vec{1}^T\right)^{-1}
	 	 \end{align*}
	 	 For convenience, we denote $L_{+}-\epsilon_2L_{-}+\frac{1}{n+1}\vec{1}\vec{1}^T$ by $B$ and $L_{+}-\epsilon_1L_{-}+\frac{1}{n+1}\vec{1}\vec{1}^T $ by $A$. Therefore,
	 	 \begin{align*} 
	 	 	B^{-1}=\left(I-\left(\epsilon_2-\epsilon_1\right)A^{-1}L_{-}\right)^{-1}A^{-1}
	 	 \end{align*}
	 	Let $\delta_{\epsilon_1}:=\frac{1}{2||A^{-1}||_2||L_{-}||_2}$ and if $0\le \epsilon_2-\epsilon_1\le\delta$:
	 	 \begin{align*}
	 	 	||(\epsilon_2-\epsilon_1)A^{-1}L_{-}||_2\le \frac{1}{2}<1
	 	 \end{align*}
	 	 So 
	 	 \begin{align*}
	 	 	\left(I-\left(\epsilon_2-\epsilon_1\right)A^{-1}L_{-}\right)^{-1}&=I+\sum_{i=1}^{\infty}\left[\left(\epsilon_2-\epsilon_1\right)A^{-1}L_{-}\right]^i\\
	 	 	B^{-1}-A^{-1}&=\sum_{i=1}^{\infty}\left[\left(\epsilon_2-\epsilon_1\right)A^{-1}L_{-}\right]^iA^{-1}
	 	 \end{align*}
	 	 In fact, for any $i\ge 1$, $\left(A^{-1}L_{-}\right)^iA^{-1}$ is positive semidefinite. When $i=2k$ is even, $\vec{x}^T\left(A^{-1}L_{-}\right)^iA^{-1}\vec{x}=\left(\left(L_{-}A^{-1}\right)^k\vec{x}\right)^TA^{-1}\left(\left(L_{-}A^{-1}\right)^k\vec{x}\right)\ge 0$. When $i=2k+1$ is odd, $\vec{x}^T\left(A^{-1}L_{-}\right)^iA^{-1}\vec{x}=\left(\left(A^{-1}L_{-}\right)^kA^{-1}\vec{x}\right)^TL_{-}\left(\left(A^{-1}L_{-}\right)^kA^{-1}\vec{x}\right)\ge 0$. Therefore,
	 	 \begin{align*}
	 	 	\Omega_{\epsilon_2}^{\sigma}(i,j)-\Omega_{\epsilon_1}^{\sigma}(i,j)&=\left(\vec{e_i}-\vec{e_j}\right)^T\left((L_{+}-\epsilon_2L_{-})^{\dagger}-(L_{+}-\epsilon_1L_{-})^{\dagger}\right)\left(\vec{e_i}-\vec{e_j}\right)\\
	 	 	&=\left(\vec{e_i}-\vec{e_j}\right)^T\left(B^{-1}-A^{-1}\right)\left(\vec{e_i}-\vec{e_j}\right)\ge 0
	 	 \end{align*} 	
	 	 Note $\delta_{\epsilon_1}$ is non-increasing with respect to $\epsilon_1$ since the second smallest eigenvalue of $A$ is non-increasing and always positive with respect to $\epsilon_1<\epsilon_0$ and $\delta_{\epsilon_0}>0$.
	 	 So $\Omega_{\epsilon}^{\sigma}(i,j)$ is non-decreasing on every closed interval whose length is no more than $\delta_{\epsilon_0}$ in $(-\infty,\epsilon_0)$. By transitivity we know $\Omega_{\epsilon}^{\sigma}(i,j)$ is non-decreasing in $(-\infty,\epsilon_0)$.
	 \end{proof}
	 
	 \section{Node and edge $\epsilon-$repelling curvature}

	  	

	 \begin{proposition}
	 	Suppose $B$ is an arbitrary symmetric matrix, then 
	 	\[ \sum_{i,j\in V}\vec{e_i}^T(\Le)B(\Le)\vec{e_j}\Omega_{\epsilon}^{\sigma}(i,j)=-2\text{tr}\left((\Le)B\right) \]
	 \end{proposition}
	 \begin{proof}
	 	Note that
	 	\[ \Omega_{\epsilon}^{\sigma}(i,j)=\vec{e_i}^T(\Le)^{\dagger}\vec{e_i}+\vec{e_j}^T(\Le)^{\dagger}\vec{e_j}-2\vec{e_i}^T(\Le)^{\dagger}\vec{e_j} \]
	 	and since $(\Le)\vec{1}=0$
	 	\begin{align*}
	 		&\sum_{i,j\in V}\vec{e_i}^T(\Le)B(\Le)\vec{e_j}\vec{e_i}^T(\Le)^{\dagger}\vec{e_i}\\
	 		=&\sum_{i\in V}\vec{e_i}^T(\Le)B(\Le)\vec{1}\vec{e_i}^T(\Le)^{\dagger}\vec{e_i}=0\\
	 		&\sum_{i,j\in V}\vec{e_i}^T(\Le)B(\Le)\vec{e_j}\vec{e_j}^T(\Le)^{\dagger}\vec{e_j}\\
	 		=&\sum_{j\in V}\vec{1}^T(\Le)B(\Le)\vec{1}\vec{e_i}^T(\Le)^{\dagger}\vec{e_i}=0
	 	\end{align*}
	 	Then we have
	 	\begin{align*}
	 		&\sum_{i,j\in V}\vec{e_i}^T(\Le)B(\Le)\vec{e_j}\Omega_{\epsilon}^{\sigma}(i,j)\\
	 		=&-2\sum_{i,j\in V}\vec{e_j}^T(\Le)B(\Le)\vec{e_i}\vec{e_i}^T(\Le)^{\dagger}\vec{e_j}\\
	 		=&-2\text{tr}\left((\Le)B(\Le)(\Le)^{\dagger}\right)\\
	 		=&-2\text{tr}\left((\Le)(\Le)^{\dagger}(\Le)B\right)\\
	 		=&-2\text{tr}\left((\Le)B\right)		
	 	\end{align*} 	
	 \end{proof}
	 
	Take $B$ as $(\Le)^{\dagger}$ into the above proposition, we get the following lemma.
	 \begin{lemma}\label{lemma2}
	 	\[ \sum_{i,j\in V}\left(w^{+}_{ij}-\epsilon w^{-}_{ij}\right)\Omega_{\epsilon}^{\sigma}(i,j)=2\left(|V(G)|-1\right) \]
	 \end{lemma}
	 
	 Inspired by the method of defining node discrete curvature which was proposed by S.Steinerberger\cite{Steinerberger2017} through a distance matrix's equilibrium solution, though $\Omega_{\epsilon}^{\sigma}$ may not be distance matrix. We define the Steinerberger's curvature on positive connected signed graph with respect to $\epsilon-$repelling matrix as node $\epsilon-$repelling curvature. 
	  
	  \begin{definition}
	  	Suppose $\Gamma=(G,\sigma,w)$ is a positive connected signed graph with $n$ vertices and $\epsilon_0$ is its consensus index. For any fixed $0<\epsilon<\epsilon_0$, we call the unique solution $\tau_{\epsilon}:V(G)\rightarrow \mathbb{R}^n$ of the following equation as node signed curvature:
	  	\[ \Omega_{\epsilon}^{\sigma}\tau_{\epsilon}=n\vec{1} \]
	  \end{definition}
	  
	  The next proposition give an explicit expression of $\tau_{\epsilon}$ in terms of the edge weight and $\epsilon-$repelling
	  cost. 
	  \begin{proposition}\label{pro3}
	  	\[ \tau_{\epsilon}(i)=\phi_{\epsilon}\left(1-\frac{1}{2}\sum_{j\sim i}(w^{+}_{ij}-\epsilon w^{-}_{ij})\Omega_{\epsilon}^{\sigma}(i,j)\right) \]
	  	where $\phi_{\epsilon}=n\vec{1}^T(\Omega_{\epsilon}^{\sigma})^{-1}\vec{1}$
	  \end{proposition}
	 
	 \begin{proof}
	 	From the matrix identity in Thm \ref{eq2}, we know
	 	\begin{align}
	 		\Omega_{\epsilon}^{\sigma}\vec{r}&=2R^2\vec{1} \label{equ3}\\
	 		-2\vec{r}\vec{1}^T+(\Le)\Omega_{\epsilon}^{\sigma}&=-2I \label{equ4} \\
	 		\vec{r}^T\vec{1}&=1
	 	\end{align}
	 	Equation (\ref{equ3}) implies 
	 	\[ \tau_{\epsilon}=\frac{n}{2R^2}\vec{r} \]
	 	and Equation (\ref{equ4}) implies
	 	\[ \vec{r}(i)=1-\frac{1}{2}\sum_{j\sim i}(w^{+}_{ij}-\epsilon w^{-}_{ij})\Omega_{\epsilon}^{\sigma}(i,j) \]
	 	By left multiplying the inverse of $\Omega_{\epsilon}^{\sigma} $ in two sides of Equation (\ref{equ3}) and then left multiplying $\vec{1}^T$, we have
	 	\[ 2R^2=\frac{1}{\vec{1}^T(\Omega_{\epsilon}^{\sigma})^{-1}\vec{1}} \]
	 	To sum up, we have
	 	\begin{align*}
	 		\tau_{\epsilon}(i)=\frac{n}{2R^2}\vec{r}(i)
	 		=n\vec{1}^T(\Omega_{\epsilon}^{\sigma})^{-1}\vec{1}\left(1-\frac{1}{2}\sum_{j\sim i}(w^{+}_{ij}-\epsilon w^{-}_{ij})\Omega_{\epsilon}^{\sigma}(i,j)\right)
	 	\end{align*}
	 \end{proof}
	 
	 Combining Lemma \ref{lemma2} and Proposition \ref{pro3}, we have the following corollary.
	 \begin{corollary}\label{coro3}
	 \[ \sum_{i\in V}\tau_{\epsilon}(i)=\phi_{\epsilon} \]	
	 \end{corollary}
	 
	 An automorphism $f$ on a signed graph $(G,\sigma,w)$ is an automorphism on its underlying graph $G$ and satisfies $\sigma_{f(x)f(y)}=\sigma_{xy}$ for any edge $(x,y)\in E(G)$. A signed graph $(G,\sigma,w)$ is called vertex transitive if any vertex $x,y$, there exists an automorphism $f$ such that $f(x)=y$.
	 
	 \begin{proposition}
	 	If Suppose $(G,\sigma,w)$ is a positive-connected and vertex transitive signed graph, the node $\epsilon-$repelling curvature is constant.
	 \end{proposition}
	 
	 \begin{proof}
	 	For any $i,j \in V$, let $f$ be the automorphism such that $f(i)=j$ and $P$ is the permutation matrix of $f$. Then we have 
	 	\begin{align*}
	 		P^T(\Le)P&=\Le\\
	 		\left(P^T(\Le)P\right)^{\dagger}&=(\Le)^{\dagger}\\
	 		P^T(\Le)^{\dagger}P&=(\Le)^{\dagger}
	 	\end{align*}
	 	Therefore, the diagonal entries of $(\Le)^{\dagger} $ are same. Indeed,
	 	\begin{align*}
	 		(\Le)^{\dagger}_{ii}&=\vec{e_i}^T(\Le)^{\dagger}\vec{e_i}\\
	 		&=\vec{e_i}^TP^T(\Le)^{\dagger}P\vec{e_i}\\
	 		&=\vec{e_j}^T(\Le)^{\dagger}\vec{e_j}\\
	 		&=(\Le)^{\dagger}_{jj}
	 	\end{align*}
	 Hence $P\zeta=\zeta$ where $\zeta$ be the vector consisting of the diagonal entries of $(\Le)^{\dagger}$.
	 As a result,
	 	\begin{align*}
	 		P^T\Omega_{\epsilon}^{\sigma}P=P^T\left(\vec{1}\zeta^T+\zeta\vec{1}^T-2(\Le)^{\dagger}\right)P&=\vec{1}\zeta^T+\zeta\vec{1}^T-2(\Le)^{\dagger}=\Omega_{\epsilon}^{\sigma}\\
	 		P^T\left(\Omega_{\epsilon}^{\sigma}\right)^{-1}P&=\left(\Omega_{\epsilon}^{\sigma}\right)^{-1}\\
	 		P\tau_{\epsilon}=nP\left(\Omega_{\epsilon}^{\sigma}\right)^{-1}P^TP\vec{1}&=n\left(\Omega_{\epsilon}^{\sigma}\right)^{-1}\vec{1}=\tau_{\epsilon}
	 	\end{align*}
	 	Thus, we know $\tau_{\epsilon}(i)=\tau_{\epsilon}(j)$ for any $i,j\in V$ and $\tau_{\epsilon}$ is constant.
	 \end{proof}
	 
	\begin{definition}
		We denote by $W_{\epsilon}(G^{\sigma})$ the sum of $\epsilon-$repelling cost of all pairs of vertices on positive connected signed graph $(G,\sigma,w)$:
		\[ W_{\epsilon}(G^{\sigma})=\sum_{i<j}\Omega_{\epsilon}^{\sigma}(i,j) \]
		called as $\epsilon-$repelling graph resistance.
	\end{definition}
	 
	 \begin{lemma}
	 	\[ W_{\epsilon}(G^{\sigma})=|V(G)|\sum_{k=2}^{|V(G)|}\frac{1}{\lambda_k^{(\epsilon)}}\]
	 	where $\left\{\lambda_k^{(\epsilon)}\right\}$ are the nonzero and non-decreasing eigenvalues of $\Le$.
	 \end{lemma}
	 \begin{proof}
	 	\begin{align*}
	 		W_{\epsilon}(G^{\sigma})&=\frac{1}{2}\sum_{i,j}\left(\vec{e_i}-\vec{e_j}\right)^T(\Le)^{\dagger}\left(\vec{e_i}-\vec{e_j}\right)\\
	 		&=|V(G)|\sum_{i\in V}\vec{e_i}^T(\Le)^{\dagger}\vec{e_i}- \vec{1}^T(\Le)^{\dagger}\vec{1}\\
	 		&=|V(G)|\text{tr}\left((\Le)^{\dagger}\right)
	 	\end{align*}
	 \end{proof}
	 
	 \begin{corollary}\label{coro2}
	 	\[ \frac{|V(G)|}{\lambda_2^{(\epsilon)}}<W_{\epsilon}(G^{\sigma})\le \frac{|V(G)|\left(|V(G)|-1\right)}{\lambda_2^{(\epsilon)}} \]
	 \end{corollary}
	 
	  Then we prove Theorem \ref{main4}  which shows a lower bound for $ \lambda_2^{(\epsilon)}$ on signed graphs with positive node signed curvature.
	  
	 \begin{proof}[Proof of Theorem \ref{main4}]
	 	\begin{align*}
	 		W_{\epsilon}(G^{\sigma})&=\frac{1}{2}\sum_{i,j}\Omega_{\epsilon}^{\sigma}(i,j)\\
	 		&\le \frac{1}{2}\sum_{i,j}\Omega_{\epsilon}^{\sigma}(i,j)\frac{\tau(j)}{K_{\epsilon}}\\
	 		&=\frac{|V(G)|^2}{2K_{\epsilon}}
	 	\end{align*}
	 	Combine the left inequality of corollary \ref{coro2} with the above inequality, we finish the proof.
	 \end{proof}
	 
	 \begin{proposition}
	 Suppose $\tau_{\epsilon}$ on positive connected signed graph $\Gamma$ is non-negative. Let $X_{\epsilon}, N_{\epsilon}$ be the maximal and minimal $\epsilon-$repelling cost among all pairs of signed resistance respectively, then
	 \[ X_{\epsilon}\le \frac{2|V(G)|}{\phi_{\epsilon}}, N_{\epsilon}\le \frac{|V(G)|}{\phi_{\epsilon}} \]
	 \end{proposition}
	 
	 \begin{proof}
	 	Assume $X_{\epsilon}=\Omega_{\epsilon}^{\sigma}(i,j)$ and $ X_{\epsilon}\le \Omega_{\epsilon}^{\sigma}(i,k)+\Omega_{\epsilon}^{\sigma}(k,j)$ for any $k\in V$. Suppose 
	 	$\mu=\frac{1}{2}\left(1_{\left\{x=i\right\}}+1_{\left\{x=j\right\}}\right)$ is a probability measure assigning equal and nonzero mass to the two points $i,j$ and zero elsewhere and $h$ is the row index such that $\Omega_{\epsilon}^{\sigma}(i,k)+\Omega_{\epsilon}^{\sigma}(k,j)$ is minimal, using the Minimax theorem in \cite[Thm 4]{Steinerberger2017} and corollary \ref{coro3}, we know
	 	\[ \frac{1}{2}X_{\epsilon}\le \frac{1}{2}\left(\Omega_{\epsilon}^{\sigma}(i,h)+\Omega_{\epsilon}^{\sigma}(h,j)\right)\le \frac{|V(G)|}{\phi_{\epsilon}} \]
	 	Therefore, we finish the proof for the upper bound of $X_{\epsilon}$. Suppose $N_{\epsilon}=\Omega_{\epsilon}^{\sigma}(k,l)$, let $\gamma=1_{\left\{x=l\right\}}$ be the probability measure only supported at vertex $l$. Similarly we have the upper bound for $N_{\epsilon}$.	 	
	 \end{proof}
	 
	 \begin{definition}
	 	Define $\vartheta_{\epsilon}$ as the edge $\epsilon-$repelling curvature on $(G,\sigma,w)$ as follows:
	 	\[ \vartheta_{\epsilon}(i,j):=\frac{2(\tau_{\epsilon}(i)+\tau_{\epsilon}(j))}{\Omega_{\epsilon}^{\sigma}}+(1+\epsilon)\Lambda_{\epsilon}(i,j)\]
	 	where \[\Lambda_{\epsilon}(i,j)=d_i^{-}\left(1-\frac{\sum_{k\in V}\Om(j,k)p^{-}_{ik}}{\Om(i,j)}\right)+d_j^{-}\left(1-\frac{\sum_{k\in V}\Om(i,k)p^{-}_{jk}}{\Om(i,j)}\right) \]
	 	and $p^{-}_{ik}=\frac{w_{ik}^{-}}{d_i^{-}}, p^{-}_{jk}=\frac{w_{jk}^{-}}{d_j^{-}}$ are the probability of random walk on negative edges adjacent to $i,j$.
	 \end{definition}
	 
	 \begin{remark}
	 	 For a pair of adjacent vertices $i\sim j$, the term $\sum_{k\in V}\Om(j,k)p^{-}_{ik} $ can be regarded as signed resistance between $j$ and the "effective negative neighbor" of $i$ and the ratio difference $1-\frac{\sum_{k\in V}\Om(j,k)p^{-}_{ik}}{\Om(i,j)}$ measures how the negative edges influence the $\epsilon-$signed resistance.
	 \end{remark}
	 
	 \begin{lemma}\label{lem4}
	 	Assume $N_t, M_t$ are random nodes distributed according to probability distribution $\rho_{t,i}:=\exp^{-Qt}\vec{e_i}$ and $\rho_{t,j}:=\exp^{-Qt}\vec{e_j}$ where $Q=L_{+}+L_{-}$ is the original Laplacian, then
	 	\[ \lim\limits_{t\rightarrow 0}\frac{1}{t}\left(1-\frac{\mathbb{E}(\Om(N_t,M_t))}{\Om(i,j)}\right)=\vartheta_{\epsilon}(i,j) \]
	 \end{lemma}
	 \begin{proof}
	 	From Equation (\ref{equ4}), we have 
	 	\[ \Om(\Le)+(\Le)\Om=2\left(\vec{1}\tau_{\epsilon}^T+\tau_{\epsilon}\vec{1}^T-2I\right) \]
	 	Since $Q=\Le+(1+\epsilon)L_{-}$, we can compute
	 	\begin{align*}
	 		\Om Q+Q\Om=2\left(\vec{1}\tau_{\epsilon}^T+\tau_{\epsilon}\vec{1}^T-2I\right)+(1+\epsilon)(\Om L_{-}+L_{-}\Om)
	 	\end{align*}
	 	Next we write $\mathbb{E}(\Om(N_t,M_t)) $ as follows:
	 	\begin{align*}
	 		\mathbb{E}(\Om(N_t,M_t))&=\sum_{k,l\in V}\vec{e_i}^T\exp^{-Qt}\vec{e_k}\Om(k,l)\vec{e_l}^T\exp^{-Qt}\vec{e_j}\\
	 		&=\vec{e_i}^T\exp^{-Qt}\Om\exp^{-Qt}\vec{e_j}\\
	 		&=\vec{e_i}^T(I-tQ+O(t^2))\Om(I-tQ+O(t^2))\vec{e_j}\\
	 		&=\Om(i,j)-2t\left(\tau_{\epsilon}(i)+\tau_{\epsilon}(j)\right)-t\left(1+\epsilon\right)\Om(i,j)(d_i^{-}+d_{j}^{-})\\
	 		&+t\left(1+\epsilon\right)\left(\sum_{k\in V}\Om(i,k)w_{kj}^{-}+\sum_{l\in V}\Om(j,l)w_{li}^{-}\right)+O(t^2)
	 	\end{align*}
	 	Therefore,
	 	\begin{align*}
	 		&\lim\limits_{t\rightarrow 0}\frac{1}{t}\left(1-\frac{\mathbb{E}(\Om(N_t,M_t))}{\Om(i,j)}\right)\\
	 		=&\frac{2(\tau_{\epsilon}(i)+\tau_{\epsilon}(j))}{\Omega_{\epsilon}^{\sigma}}+(1+\epsilon)(d_i^{-}+d_{j}^{-})-\left(1+\epsilon\right)\left(\sum_{k\in V}\frac{\Om(i,k)}{\Om(i,j)}w_{kj}^{-}+\sum_{l\in V}\frac{\Om(j,l)}{\Om(i,j)}w_{li}^{-}\right)\\
	 		=&\vartheta_{\epsilon}(i,j)
	 	\end{align*}
	 \end{proof}
	 
	Let $\kappa^{(LLY)}$ be the Lin-Lu-Yau Ricci curvature\cite{lin2011ricci} with respect to $\epsilon-$repelling cost $\Om(i,j)$ on $G$. That is,
	  \[ \kappa^{(LLY)}(i,j)=\lim\limits_{\alpha\rightarrow 0}\frac{1}{\alpha}\left(1-\frac{W\left(m_i^{\alpha},m_j^{\alpha}\right)}{\Om(i,j)}\right) \] 
	 where $m_i^{\alpha}, m_j^{\alpha}$ are the probability distributions of lazy random walk on $G$, defined as follows:
	 \begin{align*}
	 	m_i^{\alpha}(l)=
	 	\begin{cases}
	 		&1-\alpha d_i \quad \text{if}\quad l=i\\
	 		&\alpha w_{il}\quad \text{if}\quad l\sim i\\
	 		&0\quad \text{otherwise}
	 	\end{cases}
	 	\quad
	 	m_j^{\alpha}(l)=
	 	\begin{cases}
	 		&1-\alpha d_j \quad \text{if}\quad l=j\\
	 		&\alpha w_{jl} \quad \text{if}\quad l\sim j\\
	 		&0\quad \text{otherwise}
	 	\end{cases}
	 \end{align*}
	 You can find $m_i^{\alpha}=(I-\alpha Q)\vec{e_i}, m_j^{\alpha}=(I-\alpha Q)\vec{e_j}$. The optimal transport cost $$W(m_i^{\alpha},m_j^{\alpha})=\inf\limits_{\Pi}\sum_{k,l\in V}\Pi(k,l)\Om(k,l)$$ is the minimal with respect to the transport plans $\Pi$ satisfying $$ \sum_{k\in V}\Pi(k,l)=m_j^{\alpha}(l), \sum_{l\in V}\Pi(k,l)=m_i^{\alpha}(k)$$
	  Then using the above lemma, we can compare $\kappa^{(LLY)}(i,j)$ with $\vartheta_{\epsilon}(i,j)$.

	 \begin{proof}[Proof of Theorem \ref{thmv}]
	 	Obviously, $A(k,l):=m_i^{\alpha}(k)m_j^{\alpha}(l) $ is a transport plan satisfying 
	 	$$ \sum_{k\in V}A(k,l)=m_j^{\alpha}(l)\quad \sum_{l\in V}A(k,l)=m_i^{\alpha}(k)$$ Thus, we have
	 	\[ W(m_i^{\alpha},m_j^{\alpha})\le \sum_{k,l\in V}A(k,l)\Om(k,l)=\vec{e_i}^T(I-Qt)\Om(I-Qt)\vec{e_j} \]
	 	Recall \[ \mathbb{E}(\Om(N_t,M_t))=\vec{e_i}^T(I-tQ)\Om(I-tQ)\vec{e_j}+O(t^2) \]
	 	Then we know 
	 	\[ \lim\limits_{t\rightarrow 0}\frac{1}{t}\left(1-\frac{\mathbb{E}(\Om(N_t,M_t))}{\Om(i,j)}\right)\le \lim\limits_{t\rightarrow 0}\frac{1}{t}\left(1-\frac{W(m_i^{\alpha},m_j^{\alpha})}{\Om(i,j)}\right)=\kappa^{(LLY)}(i,j) \]
	 	By Lemma \ref{lem4}, we have
	 	\[  \vartheta_{\epsilon}(i,j)\le \kappa^{(LLY)}(i,j) \]
	 \end{proof}
	 
	 The Lin-Lu-Yau curvature $\kappa^{(LLY)}$ is based on the lazy random walk whose probability transition matrix is $P_t:=I-tQ$. The next lemma is concerned with the mixing rate of the lazy random walk.
	 \begin{lemma}\label{mix}
	 	Assume $G$ is finite with $|V|=n$ and $Q=D-A$ is its Laplacian. Let $d_{\text{max}}$ be the maximal degree of $G$. Then for any $0<t<\frac{1}{2d_{\text{max}}}$ and $n\in \mathbb{N}$, 
	 	\[ ||P_t^nf-\bar{f}||_2\le \rho^n||f||_2 \: ,\forall f:V\rightarrow \mathbb{R}\]
	 	where $\bar{f}=\frac{\sum_{x\in V}f(x)}{n}$, $\rho=1-t\mu_2$ is called the mixing rate of lazy random walk and $\mu_2$ is the minimal nonzero eigenvalue of $Q$.
	 \end{lemma}
	 
	 \begin{proof}
	 	It's known that the maximal eigenvalue of $Q$ is no more than $2d_{\text{max}}$. So the eigenvalues of $P_t$ is positive if $0<t<\frac{1}{2d_{\text{max}}}$. Assume $\alpha_1\le \alpha_2\le\cdots\le \alpha_n=1$ are the eigenvalues of $P_t$ and $\{f_k\}_{k=1}^n$ are the orthonormal eigenfunctions satisfying $P_tf_k=\alpha_kf_k$. Note $f_n$ is constant and $f_n\equiv \frac{1}{\sqrt{n}}$. For any $f:V\rightarrow \mathbb{R}$, we write it as:
	 	\[ f=\sum_{k=1}^{n}(f,f_k)f_k \]
	 	Thus,
	 	\begin{align*}
	 		P_tf=\sum_{k=1}^{n}\alpha_k(f,f_k)f_k
	 	\end{align*}
	 	Then
	 	\begin{align*}
	 		P_t^nf&=\sum_{k=1}^{n}\alpha_k^n(f,f_k)f_k\\
	 		&=\frac{\sum_{x\in V}f(x)}{n}+\sum_{k=1}^{n-1}\alpha_k^n(f,f_k)f_k\\
	 		&=\bar{f}+\sum_{k=1}^{n-1}\alpha_k^n(f,f_k)f_k
	 	\end{align*}
	 	Therefore,
	 	\begin{align*}
	 	||P_t^nf-\bar{f}||_2^2&=||\sum_{k=1}^{n-1}\alpha_k^n(f,f_k)f_k||_2^2\\
	 	&=\sum_{k=1}^{n-1}\alpha_k^{2n}(f,f_k)^2\\
	 	&\le \alpha_{n-1}^{2n}||f||_2^2
	 	\end{align*}
	 	equivalent to
	 	\[ ||P_t^nf-\bar{f}||_2\le \alpha_{n-1}^n||f||_2=(1-t\mu_2)^n||f||_2 \]
	 \end{proof}
	 
	 \begin{proof}[Proof of Theorem \ref{main5}]
	 	As a result of Theorem \ref{thmv}, we have
	 	\[ \kappa^{(LLY)}(i,j)\ge \vartheta_{\epsilon}(i,j)\ge k_{\epsilon} \]
	 	Since $G$ is finite, for any $\delta>0$, there exists $t_0>0$ satisfying:
	 	 \[ \frac{1}{t}\left(1-\frac{W(m_i^{t},m_j^{t})}{\Om(i,j)}\right)\ge (1-\delta)k_{\epsilon}, \quad \forall t\in (0,t_0) \]
	 	equivalent to
	 	\[ W(m_i^{t},m_j^{t})\le \left(1-\left(1-\delta\right)k_{\epsilon}t\right)\Om(i,j), \quad \forall t\in (0,t_0) \]
	 	From Kantorovich duality in \cite{villani2008optimal}, we know
	 	\[ W(m_i^{t},m_j^{t})=\sup\limits_{\Phi_{\epsilon}}\left\{ \sum_{x\in V} \phi(x)m_i^{t}(x)+\sum_{y\in V} \psi(y)m_j^{t}(y)\right\} \] 
	 	where the set $\Phi_{\epsilon}$ consists of all the pairs $(\phi,\psi)$ satisfying
	 	\[ \phi(x)+\psi(y)\le \Om(x,y), \quad \forall x,y \in V \]
	 	For any $k>0$ and assume $\phi:V\rightarrow \mathbb{R}$ satisfies $|\phi(x)-\phi(y)|\le k\Om(x,y), \quad \forall x,y \in V$, since $(\frac{\phi}{k},-\frac{\phi}{k}), (-\frac{\phi}{k},\frac{\phi}{k})\in \Phi_{\epsilon}$, 
	 	\begin{align*}
	 		W(m_i^{t},m_j^{t})&\ge \frac{1}{k}|\sum_{x\in V}\phi(x)\left(m_i^{t}(x)-m_j^{t}(x)\right)|\\
	 		&=|\left(I-tQ\right)\phi(i)-\left(I-tQ\right)\phi(j)|
	 	\end{align*}
	 Thus,
	 \[ |\left(I-tQ\right)\phi(i)-\left(I-tQ\right)\phi(j)|\le k\left(1-\left(1-\delta\right)k_{\epsilon}t\right)\Om(i,j) \]
	 Then for any $n\ge 1$,
	 \begin{align*}
	 	|\left(I-tQ\right)^n\phi(i)-\left(I-tQ\right)^n\phi(j)|&\le k\left(1-\left(1-\delta\right)k_{\epsilon}t\right)^n\Om(i,j)\\
	 	\sqrt[n]{|\left(I-tQ\right)^n\phi(i)-\left(I-tQ\right)^n\phi(j)|}&\le \sqrt[n]{k\Om(i,j)}\left(1-\left(1-\delta\right)k_{\epsilon}t\right)
	 \end{align*}
	 Let $n\rightarrow \infty$, we have
	 \[ \limsup_{n\rightarrow \infty}\sqrt[n]{|\left(I-tQ\right)^n\phi(i)-\left(I-tQ\right)^n\phi(j)|}\le 1-\left(1-\delta\right)k_{\epsilon}t \]
	 On the other hand, consider the mixing rate of $I-tQ$ and use Lemma \ref{mix}, for any $\phi:V\rightarrow \mathbb{R}$ we have :
	\begin{align*}
		|\left(I-tQ\right)^n\phi(i)-\left(I-tQ\right)^n\phi(j)|&\le |\left(I-tQ\right)^n\phi(i)-\bar{\phi}|+|\left(I-tQ\right)^n\phi(j)-\bar{\phi}|\\
		&\le ||\left(I-tQ\right)^n\phi-\bar{\phi}||\\
		&\le (1-t\mu_2)^n||\phi||_2
	\end{align*}
	Then
	\[ \limsup_{n\rightarrow \infty}\sqrt[n]{|\left(I-tQ\right)^n\phi(i)-\left(I-tQ\right)^n\phi(j)|}\le 1-t\mu_2\]
	Since $1-t\mu_2$ is the mixing rate, we have
	\[ 1-t\mu_2\le 1-\left(1-\delta\right)k_{\epsilon}t \]
	equivalent to
	\[ \mu_2\ge \left(1-\delta\right)k_{\epsilon} \]
 Let $\delta\rightarrow 0$, we have
 \[ \mu_2\ge k_{\epsilon} \]
	 \end{proof}
	 \section{Some examples}
	 \subsection{positive connected signed cycle}
	 Consider positive connected signed cycle $(C_n,\sigma_n)$ with simple weight, there exist only two cases. The first is trivial, that is $\sigma_n\equiv 1$ which means no negative edges. Since $L_{-}$ is zero, $\Om(\cdot,\cdot)$ is equal to resistance distance on the underlying graph $C_n$ for any $\epsilon$.  The other is that $(C_n,\sigma_n)$ consists of the positive subgraph $P_n$ and the negative subgraph $P_2$ where $P_n$ is path graph with $n$ vertices. 
	 
	 We assume $\sigma_n$ takes $-1$ at the edge $(1,n)$ and $1$ at other edges and compute the consensus index $\epsilon_0$ as well as the signed node curvature and edge curvature for several values of $\epsilon $ smaller than $\epsilon_0$ when $n=3, 4$.
	 
	 \paragraph*{\textbf{The consensus index of $(C_3,\sigma_3)$ is $\epsilon_0=0.5$}:} $\tau_{\epsilon}(1)=\tau_{\epsilon}(3)$ since vertices $1$ and $3$ can be transitive to each other. We can get  $\Lambda_{\epsilon}(1,2)=\Lambda_{\epsilon}(2,3)=0$ and $\Lambda_{\epsilon}(1,3)=2$ easily. 
	 \begin{table}[htbp]
	 	\centering
	 	\caption{$\tau_{\epsilon}$ and $\vartheta_{\epsilon}$ when $\epsilon=0.2, 0.3 0.449$}
	 	\begin{tabularx}{0.8\textwidth}{@{} c c c c c @{}}
	 		\toprule
	 		\textbf{$\epsilon$}   & \textbf{$\tau_{\epsilon}(2)$} &  \textbf{$\tau_{\epsilon}(1)=\tau_{\epsilon}(3)$}        & \textbf{$\vartheta_{\epsilon}(12)=\vartheta_{\epsilon}(23)$}& \textbf{$\vartheta_{\epsilon}(13)$}  \\ 
	 		\midrule
	 		0.2           & -0.5625           & 1.125 &  0.844  & 3.7501 \\
	 		0.3           & -0.7347           & 0.8571     &  0.1399  & 3.2857 \\
	 		0.4999           & -0.0012           & 0.0006     &  $\approx$0  & 2.9998 \\
	 		\bottomrule
	 	\end{tabularx}
	 \end{table}
	 
	 \paragraph*{\textbf{The consensus index of $(C_4,\sigma_4)$ is $\epsilon_0=0.33329$}:} We know $\tau_{\epsilon}(2)=\tau_{\epsilon}(3), \tau_{\epsilon}(1)=\tau_{\epsilon}(4)$ by transitivity. Besides, we have
	 \begin{align*}
	 	\Lambda_{\epsilon}(2,3)&=0\\
	 	\Lambda_{\epsilon}(1,2)=1-\frac{\Om(2,4)}{\Om(1,2)}&=1-\frac{\Om(2,4)}{\Om(1,2)}=\Lambda_{\epsilon}(3,4)\\
	 	\Lambda_{\epsilon}(1,4)&=2
	 \end{align*}
	 \begin{table}[htbp]
	 	\centering
	 	\caption{$\tau_{\epsilon}$ and $\vartheta_{\epsilon}$ when $\epsilon=0.1, 0.2, 0.3332$}
	 	\begin{tabularx}{0.9\textwidth}{@{} c c c c c c@{}}
	 		\toprule
	 		\textbf{$\epsilon$}   & \textbf{$\tau_{\epsilon}(2)=\tau_{\epsilon}(3)$} &  \textbf{$\tau_{\epsilon}(1)=\tau_{\epsilon}(4)$}        & \textbf{$\vartheta_{\epsilon}(12)=\vartheta_{\epsilon}(34)$}& \textbf{$\vartheta_{\epsilon}(23)$}&  \textbf{$\vartheta_{\epsilon}(14)$}\\ 
	 		\midrule
	 		0.1        & -0.2569           & 1.156 &  0.1985  & -0.8991&3.2789 \\
	 		0.2          & -0.4211           & 0.8421    &  -1.4387  & -1.1229&2.8491 \\
	 		0.3332           & -0.0012           & 0.0012    &  -3.9964  & $\approx$0&2.6664 \\
	 		\bottomrule
	 	\end{tabularx}
	 \end{table}
	 
	 \subsection{positive connected signed complete graph}
	The positive subgraph of a positive connected signed complete graph $(K_n,\sigma_n)$ must contain some spanning tree and the number of the complete graph's spanning trees is $n^{n-2}$. So you can image the number of non-isomorphic $(K_n,\sigma_n)$ is also large when $n$ is large. Here, we enumerate several non-isomorphic signed graph structures of $K_4$ and also compute the node and edge curvatures on them.
	
	\paragraph*{\textbf{The case of one negative edge:}} Let $\sigma_4$ take $-1$ at $(1,4)$ and $1$ at other edges. The consensus index $\epsilon_0$ is $0.9999$. Take $\epsilon$ as $0.5$. Then we compute $\tau_{\epsilon}(1)=\tau_{\epsilon}(4)=2.4242, \tau_{\epsilon}(2)=\tau_{\epsilon}(3)=-0.4848$. Since $\Lambda_{\epsilon}(\cdot,\cdot) $ takes nonzero values only at $(1,4)$ and $\Lambda_{\epsilon}(1,4)=2$, we compute $\vartheta_{\epsilon}(1,2)=\vartheta_{\epsilon}(3,4)=\vartheta_{\epsilon}(4,2)=\vartheta_{\epsilon}(3,1)=4.4329, \vartheta_{\epsilon}(2,3)=-3.8784, \vartheta_{\epsilon}(1,4)=7.8484$.

	\paragraph*{\textbf{The case of two negative edges:}} 
	\begin{itemize}
		\item Let $\sigma_4$ take $-1$ at $(1,3), (2,4)$ and $1$ at other edges. The consensus index $\epsilon_0$ is $1$. Take $\epsilon$ as $0.5$. We compute $\tau_{\epsilon}(1)=\tau_{\epsilon}(2)=\tau_{\epsilon}(3)=\tau_{\epsilon}(4)=0.8889$. Since $\Lambda_{\epsilon}(1,3)=\Lambda_{\epsilon}(2,4)=2$ and $\Lambda_{\epsilon}$ is zero on other edges, we compute $\vartheta_{\epsilon}(1,2)=\vartheta_{\epsilon}(2,3)=\vartheta_{\epsilon}(3,4)=\vartheta_{\epsilon}(1,4)=2.8445;\vartheta_{\epsilon}(1,3)=\vartheta_{\epsilon}(2,4)=4.7778$.
		
		\item Let $\sigma_4$ take $-1$ at $(1,3), (1,4)$ and $1$ at other edges. The consensus index $\epsilon_0$ is $0.333$. Take $\epsilon$ as $0.1$. We compute $\tau_{\epsilon}(1)=1.7436, \tau_{\epsilon}(2)=-1.1454, \tau_{\epsilon}(3)=\tau_{\epsilon}(4)=1.1819$. Since $\Lambda_{\epsilon}(1,3)=\Lambda_{\epsilon}(1,4)=2.7021, \Lambda_{\epsilon}(2,3)=\Lambda_{\epsilon}(2,4)=-0.7286, \Lambda_{\epsilon}(1,2)=0.843, \Lambda_{\epsilon}(3,4)=-4.7139$ , we compute $\vartheta_{\epsilon}(1,3)=\vartheta_{\epsilon}(1,4)=5.4994,    \vartheta_{\epsilon}(2,3)=\vartheta_{\epsilon}(2,4)=-0.7033, \vartheta_{\epsilon}(1,2)=1.8578, \vartheta_{\epsilon}(3,4)=1.6693$.
	\end{itemize}
	
	\paragraph*{\textbf{The case of three negative edges:}}  
	\begin{itemize}
		\item Let $\sigma_4$ take $1$ at $(1,3), (1,4), (2,3)$ and $-1$ at other edges. The consensus index $\epsilon_0$ is $0.1715$. Take $\epsilon$ as $0.1$. Since $\Lambda_{\epsilon}(1,2)=\Lambda_{\epsilon}(3,4)=-6.1347, \Lambda_{\epsilon}(1,4)=\Lambda_{\epsilon}(2,3)=2.6556, \Lambda_{\epsilon}(1,3)=3.9994, \Lambda_{\epsilon}(2,4)=0.3492$, we can compute:
		\begin{align*}
			\tau_{\epsilon}(1)=\tau_{\epsilon}(3)=0.8344, &\tau_{\epsilon}(2)=\tau_{\epsilon}(4)=-0.3457\\
			\vartheta_{\epsilon}(1,2)=\vartheta_{\epsilon}(3,4)=-6.0546,& \vartheta_{\epsilon}(1,4)=\vartheta_{\epsilon}(2,3)=3.16\\ \vartheta_{\epsilon}(1,3)=4.8712, & \vartheta_{\epsilon}(2,4)=-0.4258
		\end{align*}
		\item Let $\sigma_4$ take $-1$ at $(1,2), (1,3), (1,4)$ and $1$ at other edges. The consensus index $\epsilon_0$ is $0.3333$. Take $\epsilon$ as $0.1$. Since $\Lambda_{\epsilon}(1,2)=\Lambda_{\epsilon}(1,3)=\Lambda_{\epsilon}(1,4)=0, \Lambda_{\epsilon}(2,3)=\Lambda_{\epsilon}(2,4)=\Lambda_{\epsilon}(3,4)=2$, we can compute:
		\begin{align*}
			\tau_{\epsilon}(1)=-1.4979, &\tau_{\epsilon}(2)=\tau_{\epsilon}(3)=\tau_{\epsilon}(4)=1.037\\
			\vartheta_{\epsilon}(1,2)=&\vartheta_{\epsilon}(1,3)=\vartheta_{\epsilon}(1,4)=-0.717\\ \vartheta_{\epsilon}(2,3)=&\vartheta_{\epsilon}(2,4)=\vartheta_{\epsilon}(3,4)=3.6518
		\end{align*}
		
	\end{itemize}


\begin{thebibliography}{99}
	\bibitem{heider1946attitudes}
	Fritz Heider,
	\textit{Attitudes and cognitive organization},
	\textit{The Journal of Psychology},
	vol.~21, no.~1, pp.~107--112, 1946.
	\bibitem{cartwright1956structural}
	Dorwin Cartwright and Frank Harary,  
	\textit{Structural balance: a generalization of Heider's theory},  
	\textit{Psychological Review},  
	vol.~63, no.~5, pp.~277--293, 1956.
	\bibitem{ising1925beitrag}
	Ernst Ising,  
	\textit{Beitrag zur Theorie des Ferromagnetismus},  
	\textit{Zeitschrift für Physik},  
	vol.~31, no.~1, pp.~253--258, 1925.
	\bibitem{zaslavsky1982signed}
	Thomas Zaslavsky,  
	\textit{Signed graphs},  
	\textit{Discrete Applied Mathematics},  
	vol.~4, no.~1, pp.~47--74, 1982.
	\bibitem{jiang2023spherical}
	Zilin Jiang, Jonathan Tidor, Yuan Yao, Shengtong Zhang, and Yufei Zhao.
	\newblock Spherical two-distance sets and eigenvalues of signed graphs.
	\newblock {\em Combinatorica}, 43(2):203--232, 2023.
	\bibitem{bilu2006lifts}
	Yonatan Bilu and Nathan Linial.
	\newblock Lifts, discrepancy and nearly optimal spectral gap.
	\newblock {\em Combinatorica}, 26(5):495--519, 2006.
	\bibitem{lichnerowicz1958book}
	A.~Lichnerowicz, \emph{Géométrie des groupes de transformations}, vol.~3 of \emph{Travaux et Recherches Mathématiques}. Paris: Dunod, 1958.
	\bibitem{liu2019curvature}
	S. Liu, F. Münch, and N. Peyerimhoff,
	\newblock Curvature and higher order Buser inequalities for the graph connection Laplacian,
	\newblock {\em SIAM Journal on Discrete Mathematics}, 33(1):257--305, 2019.
	
	\bibitem{bhatia2013matrix}
	R. Bhatia, \textit{Matrix Analysis}, vol. 169, Springer Science \& Business Media, 2013.
	\bibitem{shi2019dynamics}
	G. Shi, C. Altafini, and J. S. Baras, 
	\textit{Dynamics over signed networks},  
	SIAM Review, vol. 61, no. 2, pp. 229--257, 2019.
	\bibitem{olfati2004consensus}
	R. Olfati-Saber and R. M. Murray, 
	\textit{Consensus problems in networks of agents with switching topology and time-delays},  
	IEEE Transactions on Automatic Control, vol. 49, no. 9, pp. 1520--1533, 2004.
	\bibitem{chen2016definiteness}
	Y.~Chen, S.~Z.~Khong, and T.~T.~Georgiou, ``On the definiteness of graph Laplacians with negative weights: Geometrical and passivity-based approaches,'' in \emph{Proceedings of the 2016 American Control Conference (ACC)}, IEEE, 2016, pp. 2488--2493.
	\bibitem{zelazo2014definiteness}
	D.~Zelazo and M.~B{\"u}rger, ``On the definiteness of the weighted Laplacian and its connection to effective resistance,'' in \emph{Proceedings of the 53rd IEEE Conference on Decision and Control}, IEEE, 2014, pp.~2895--2900.
	\bibitem{harary1953notion}
	F.~Harary, ``On the notion of balance of a signed graph,'' \emph{Michigan Mathematical Journal}, vol.~2, no.~2, pp.~143--146, 1953.
	\bibitem{chung2014ranking}  
	F. Chung, W. Zhao, and M. Kempton,  
	"Ranking and sparsifying a connection graph,"  
	\textit{Internet Mathematics},  
	vol. 10, no. 1-2, pp. 87--115, 2014.
	\bibitem{cloninger2024random}  
	A.~Cloninger, G.~Mishne, A.~Oslandsbotn, S.~J. Robertson, Z.~Wan, and Y.~Wang,  
	``Random walks, conductance, and resistance for the connection graph Laplacian,''  
	\emph{SIAM Journal on Matrix Analysis and Applications}, vol.~45, no.~3, pp.~1541--1572, 2024.
	\bibitem{devriendt2022effective}
	K.~Devriendt, ``Effective resistance is more than distance: Laplacians, simplices and the Schur complement,'' 
	\emph{Linear Algebra and its Applications}, vol.~639, pp.~24--49, 2022.
	\bibitem{fiedler2011matrices}
	M.~Fiedler, \emph{Matrices and Graphs in Geometry}, vol.~139.  
	Cambridge University Press, 2011.
	
	\bibitem{devriendt2024graph}
	K.~Devriendt, A.~Ottolini, and S.~Steinerberger,
	\newblock Graph curvature via resistance distance,
	\newblock {\em Discrete Appl. Math.}, \textbf{348} (2024), 68--78.
	\bibitem{Steinerberger2017}
	Steinerberger, S. (2017). Curvature on graphs via equilibrium measures. \textit{Journal of Graph Theory}, \textbf{95}(1), 1–19.
	\bibitem{dev2022}
	Karel Devriendt and Renaud Lambiotte, "Discrete curvature on graphs from the effective resistance," \textit{Journal of Physics: Complexity}, vol. 3, no. 2, pp. 025008, 2022.
	\bibitem{lin2011ricci}
	Yong Lin, Linyuan Lu, and Shing-Tung Yau, 
	\textit{Ricci curvature of graphs}, 
	\textit{Tohoku Mathematical Journal, Second Series}, vol. 63, no. 4, pp. 605--627, 2011.
	\bibitem{villani2008optimal}
	C.~Villani,
	\textit{Optimal Transport: Old and New},
	Grundlehren der mathematischen Wissenschaften, vol.~338, Springer, 2008.
	
	\end{thebibliography}
\end{document}